\documentclass[11pt]{amsart}
\usepackage{preamble}

\begin{document}

\title[Unique Optima of the Delsarte Linear Program]{Unique Optima of the Delsarte Linear Program}
\author[Rupert Li]{Rupert Li}
\date{\today}

\maketitle

\begin{abstract}
The Delsarte linear program is used to bound the size of codes given their block length $n$ and minimal distance $d$ by taking a linear relaxation from codes to quasicodes.
We study for which values of $(n,d)$ this linear program has a unique optimum: while we show that it does not always have a unique optimum, we prove that it does if $d>n/2$ or if $d \leq 2$.
Introducing the Krawtchouk decomposition of a quasicode, we prove there exist optima to the $(n,2e)$ and $(n-1,2e-1)$ linear programs that have essentially identical Krawtchouk decompositions, revealing a parity phenomenon among the Delsarte linear programs.
We generalize the notion of extending and puncturing codes to quasicodes, from which we see that this parity relationship is given by extending/puncturing.
We further characterize these pairs of optima, in particular demonstrating that they exhibit a symmetry property, effectively halving the number of decision variables.
\end{abstract}

\section{Introduction}

The practical problem of communicating over a noisy channel, whose study was initiated by Hamming \cite{hamming1950error}, can be modeled by the problem of choosing as many words as possible for our code such that no two words are less than $d$ distance apart.
Let $\F_q$ denote an alphabet with $q$ elements, and let $|x-y|$ denote the \emph{Hamming distance} between words $x,y \in \F_q^n$, i.e., the number of indices for which the words have different letters.
This notation suggests $\F_q$ has the additional structure of a finite field, which is often used to describe certain particularly elegant codes, but is not necessary for the definition of the problem.
However, we will pick a distinguished word $0 \in \F_q^n$ and define the \emph{weight} of $x$, denoted $|x|$, to be $|x-0|$.

A \emph{code} of length $n$ is then simply a nonempty subset $\mathcal{C}\subseteq \F_q^n$.
Its \emph{minimal distance} $d$ is given by
\[ d = \min_{\substack{x,y \in \mathcal{C} \\ x \neq y}} |x-y|. \]
The value of $d$ corresponds to how error-resistant the code is: if we only use words in $\mathcal{C}$ rather than arbitrary words in $\F_q^n$, then if up to $d-1$ letters are changed (perhaps due to an error in storage, or due to communication over a noisy transmission channel), it is impossible for one word to be changed to another valid word, so any such modified word will be detected as a word not in the code.
Thus, we say that the code \emph{detects} $d-1$ errors, and similarly it \emph{corrects} $\floor{(d-1)/2}$ errors, in the sense that if at most $\floor{(d-1)/2}$ errors occur, the original word can be determined.

Immediately, a tradeoff manifests itself: in order to transmit more information per word, one would like the code to have larger \emph{size} $|\mathcal{C}|$, i.e., the number of words in the code, but this causes the minimal distance to decrease, reducing the code's resistance to errors.
Hence, the basic question of codes is, given length $n$, alphabet size $q$, and a lower bound for the minimal distance $d$, what is the maximum possible size of the code?

One may view this problem as the sphere-packing problem in Hamming space $\F_q^n$ rather than Euclidean space $\R^n$.
The sphere-packing problem is the problem of how to most densely pack unit balls into $\R^n$ without overlap, or equivalently the problem of picking ``as many" points in $\R^n$ as possible such that the minimal distance $d$ is at least 2, where ``as many" refers to the density of points contained in a closed ball as the radius of this ball goes to infinity.
It has been solved for $n \leq 3$ (see Hales \cite{hales2000cannonballs,hales2005proof}), as well as famously for $n=8$ by Viazovska \cite{viazovska2017sphere} and for $n=24$ by Cohn, Kumar, Miller, Radchenko, and Viazovska \cite{cohn2017sphere}.
One complication that arises in Hamming space compared to Euclidean space is that due to the continuity of $\R^n$, the choice of the lower bound on minimal distance $d$ does not matter as the space can be dilated appropriately, but this is not the case in Hamming space, where the choice of $d$ nontrivially changes the structure of the largest codes.
For more background on sphere-packing, error-correcting codes, and their connections, we refer readers to Conway and Sloane \cite{conway2013sphere}.

We now introduce concepts that will be useful in addressing the question of the largest possible code given $n$ and $d$, where we fix $q=2$ for the remainder of the paper.
The $j$-th \emph{Krawtchouk polynomial} is defined by
\[ K_j(i;n) = \sum_{k=0}^j (-1)^k \binom{i}{k}\binom{n-i}{j-k}, \]
which we typically write as $K_j(i)$ as $n$ will be clear from the context.
We will use the convention that $\binom{a}{b}=0$ if $0 \leq b \leq a$ does not hold, so $K_j(i;n)=0$ if $j\not\in[0,n]$ or $i\not\in[0,n]$, where we define $[a,b]=\{x\in\Z\mid a\leq x \leq b\}$.
For a word $x \in \F_2^n$ of weight $i$, the Krawtchouk polynomial $K_j(i)$ is the sum of $(-1)^{\left<x,y\right>}$ over all words $y \in \F_2^n$ of weight $j$, where $\left<x,y\right>$ is the inner product of $x$ and $y$ over the finite field $\F_2$, which up to parity equals the number of indices for which $x$ and $y$ are both 1.
From this interpretation, we typically consider $i$ and $j$ for $0 \leq i,j \leq n$, so the Krawtchouk polynomials can be condensed into the \emph{Krawtchouk matrix} $K$, the $(n+1)\times(n+1)$ matrix given by $K_{ji} = K_j(i)$ for $0 \leq i,j \leq n$.
Hence, we let $K_j$ denote the $j$-th row of $K$, representing the $j$-th Krawtchouk polynomial.

The \emph{distance distribution} of a code $\mathcal{C}\subseteq\F_q^n$ is the vector $A=(A_0,\dots,A_n)$, where
\[ A_i = \frac{1}{|\mathcal{C}|} \left|\left\{(x,y)\in \mathcal{C}^2 : |x-y|=i\right\}\right| \]
for $0 \leq i \leq n$.
The normalization factor of $|\mathcal{C}|^{-1}$ ensures that $A_0=1$ and $\sum_{i=0}^n A_i = |\mathcal{C}|$.
If a lower bound $d$ on the minimal distance is specified, this corresponds to requiring $A_i = 0$ for all $i \in [d-1]$, where $[a]$ denotes the set $\{1,\dots,a\}$.
The \emph{support} of $\mathcal{C}$ is $S=\{i>0\mid A_i > 0\}$.

Delsarte \cite{delsarte1972bounds} proved that the distance distribution of any code must satisfy certain inequalities expressed in terms of Krawtchouk polynomials.
This yields the \emph{Delsarte linear program}, which, for a given pair of integers $(n,d)$ such that $1 \leq d \leq n$, is given by
\begin{align}
    \max \quad & \sum_{i=0}^n A_i \notag \\
    \text{such that} \quad & \sum_{i=0}^n A_i K_j(i) \geq 0 \quad \text{for all } j \in [n] \label{eq: Delsarte inequalities} \\
    & A_i = 0 \quad \text{for all } i \in [d-1] \notag \\
    & A_0 = 1 \notag \\
    & A_i \geq 0 \quad \text{for all } i \in [n]. \notag
\end{align}
The inequalities in \cref{eq: Delsarte inequalities} are referred to as the \emph{Delsarte inequalities}.
Any code $\mathcal{C}\subseteq \F_2^n$ whose minimal distance is at least $d$ has a distance distribution that is a feasible point of the Delsarte linear program, and as $|\mathcal{C}|=\sum_{i=0}^n A_i$, we find the optimal objective value of the Delsarte linear program is an upper bound on $|\mathcal{C}|$.
As an arbitrary feasible solution $A$ may not actually be realized as the distance distribution of a code, we refer to these points $A=(A_0,\dots,A_n)$ as \emph{quasicodes}.
Let the feasible region for this linear program, which is a convex polytope in $\R^{n-d+1}$ corresponding to variables $A_d$ through $A_n$, be denoted by $P$.

The dual of the Delsarte linear program is given by
\begin{align*}
    \min \quad & \sum_{j=0}^n c_j \binom{n}{j} \\
    \text{such that} \quad & \sum_{j=0}^n c_j K_j(i) \leq 0 \quad \text{for all } i \in [d,n] \\
    & c_j \geq 0 \quad \text{for all } j \in [n] \\
    & c_0 = 1.
\end{align*}

By complementary slackness, for any optimal quasicode $A^*$ and optimal dual solution $c^*$, for all $i \in [d,n]$, if $A^*_i > 0$ then $\sum_{j=0}^n c^*_j K_j(i) = 0$.
And for all $j \in [n]$, if $c^*_j > 0$ then $\sum_{i=0}^n A^*_i K_j(i) = 0$.

The main results of this paper comprise \cref{section: Krawtchouk decomposition of optimum,section: extending and puncturing}.
We introduce the \emph{Krawtchouk decomposition} of a quasicode $A$, which is the vector $b=(b_0,\dots,b_n)$ given by
\[ b_j = \frac{1}{\binom{n}{j}}\sum_{i=0}^n A_i K_j(i), \]
so that for all $i$,
\[ A_i = \binom{n}{i}(b_0 K_0(i) + \cdots + b_n K_n(i))/2^n. \]
We prove that there exist optima, i.e., feasible points achieving the optimal value, of the $(n,2e)$ and $(n-1,2e-1)$ Delsarte linear programs whose Krawtchouk decompositions agree on indices 0 through $n-1$, inclusive.
This reveals a previously unseen parity phenomenon in the Delsarte linear program.
Moreover, this parity phenomenon manifests from the generalization of extending and puncturing codes to quasicodes.
Two common practical operations performed on codes are extending a code by adding a parity check bit and puncturing a code by removing a bit.
We introduce the generalization of these two operations to quasicodes, and show that solutions to the $(n-1,2e-1)$ and $(n,2e)$ Delsarte linear programs can be sent to each other by extending and puncturing, while still preserving the optimal value.
Thus, extending a $(n-1,2e-1)$ optima yields a $(n,2e)$ optima, and vice versa for puncturing.
In particular, puncturing corresponds to truncating the Krawtchouk decomposition, proving the parity phenomenon.

The parity phenomenon suggests that quasicodes fundamentally possess the same structure as codes with respect to extending and puncturing.
We additionally prove a symmetry phenomenon that shows the structure of even codes persists among quasicodes.
When $d$ is even, efficient codes are typically \emph{even}, meaning the distance distribution $A$ has even support $S \subset 2\Z$.
We prove that when $d$ is even, the $(n,d)$ Delsarte linear program always has an even optimum, demonstrating that the evenness structure extends from codes to quasicodes.
The quasicode $A$ being even corresponds to the Krawtchouk decomposition $b$ being symmetric, i.e., $b_j=b_{n-j}$ for all $j$, and thus this proves our symmetry phenomenon.

In \cref{section: Krawtchouk}, we provide some preliminary properties of the Krawtchouk polynomials that will be important for later sections.
In \cref{section: uniqueness of optima}, we prove that if $d>n/2$ or if $d\leq 2$, the Delsarte linear program has a unique optimum.
We also show that the dual does not have a unique optimum in many cases, and present some examples in which the primal does not have a unique optimum.
In \cref{section: Krawtchouk decomposition of optimum}, we define the Krawtchouk decomposition of a quasicode and then present the parity and symmetry phenomena.
In \cref{section: extending and puncturing}, we generalize the notion of extending and puncturing codes to quasicodes, allowing us to prove the parity and symmetry phenomena via extending/puncturing quasicodes.

\section{Preliminaries of the Krawtchouk polynomials}\label{section: Krawtchouk}

We now introduce numerous classically known properties of the Krawtchouk polynomials that will be useful throughout the paper.

\begin{lemma}[Reciprocity of the Krawtchouk polynomials; see \protect{\cite[p.\ 152]{macwilliams1977theory}}]\label{lemma: Krawtchouk reciprocity}
For all $0 \leq i,j \leq n$, we have $\binom{n}{i}K_j(i) = \binom{n}{j}K_i(j)$.
\end{lemma}

The well-established basic properties in \cref{lemma: Krawtchouk basic properties} directly follow from the definition of the Krawtchouk polynomials and \cref{lemma: Krawtchouk reciprocity}.

\begin{lemma}\label{lemma: Krawtchouk basic properties}
The Krawtchouk polynomials satisfy the following properties, for all $0 \leq i,j \leq n$:
\begin{enumerate}
    \item $K_0(i) = 1$.
    \item $K_j(0) = \binom{n}{j}$.
    \item $K_{n-j}(i) = (-1)^i K_j(i)$.
    \item $K_j(n-i) = (-1)^j K_j(i)$.
\end{enumerate}
\end{lemma}

\begin{lemma}[Orthogonality of the Krawtchouk polynomials; see \protect{\cite[p.\ 151]{macwilliams1977theory}}]\label{lemma: Krawtchouk orthogonality}
For all $0 \leq j,k \leq n$,
\begin{align*}
    \frac{1}{2^n} \sum_{i=0}^n \binom{n}{i} K_j(i) K_k(i) = \binom{n}{j} \delta_{jk},
\end{align*}
where $\delta$ is the Kronecker delta function.
\end{lemma}
\cref{lemma: Krawtchouk orthogonality} allows you to recover the coefficients of a linear combination of Krawtchouk polynomials: if $v=(v_0,\dots,v_n)=b_0K_0 + \cdots + b_n K_n$, then
\begin{align*}
    \frac{1}{2^n}\sum_{i=0}^n \binom{n}{i}v_i K_j(i) = \binom{n}{j}b_j.
\end{align*}
It also implies $K^2=2^n I$, as
\[ (K^2)_{jk} = \sum_{i=0}^n K_j(i) K_i(k) = \sum_{i=0}^n \frac{\binom{n}{i} K_j(i) K_k(i)}{\binom{n}{k}} = 2^n \delta_{jk}. \]
In particular, $K$ is non-singular and the Krawtchouk polynomials $K_j$ are linearly independent.

The following result gives the column sums of the Krawtchouk matrix $K$.

\begin{lemma}
[\protect{\cite[p.\ 153]{macwilliams1977theory}}]
\label{lemma: Krawtchouk column-sum zero}
For all $n$,
\[ \sum_{j=0}^n K_j = \left(2^n,0,0,\dots,0\right).\]
\end{lemma}

The Krawtchouk polynomials also satisfy a three-term recurrence.
\begin{lemma}
[\protect{\cite[p.\ 152]{macwilliams1977theory}}]
\label{lemma: Krawtchouk three-term recurrence}
For all $0 \leq i,j \leq n$,
\[ (j+1)K_{j+1}(i)=(n-2i)K_j(i)-(n-j+1)K_{j-1}(i;n). \]
\end{lemma}

While typically $n$ is kept fixed, we present the following useful recurrence for Krawtchouk polynomials between block lengths $n$ and $n-1$.
\begin{lemma}[\protect{\cite[Proposition~2.1(1)]{chihara1990zeros}}]\label{lemma: Krawtchouk recurrence}
For all $n \geq 1$ and $0 \leq i\leq n-1$,
\[ K_j(i;n)=K_j(i;n-1)+K_{j-1}(i;n-1) \]
for any $0 \leq j \leq n$.
\end{lemma}

In the reverse direction, we provide another relation moving from block length $n-1$ to $n$.
\begin{lemma}[\protect{\cite[Proposition~2.1(4)]{chihara1990zeros}}]\label{lemma: Krawtchouk extension}
For all $0 \leq i,j \leq n-1$,
\[ 2K_j(i;n-1)=K_j(i;n)+K_j(i+1;n). \]
\end{lemma}

It is known that the magnitude of $K_j(i)$ over all integers $0 \leq i \leq n$ is maximized at $i=0$ and $n$ (see, for example, Dette \cite{dette1995new}).
The following result shows that $i=1$ and $n-1$ are next largest, i.e., $|K_j(i)|$ over all $1 \leq i \leq n-1$ is maximized at $i=1$ and $n-1$, except when $j=\frac{n}{2}$ so that $K_j(1)=\binom{n}{d}\frac{n-2j}{n}=0$.

\begin{lemma}[\protect{\cite[Corollary~10]{dumer2013spherically}}]\label{lemma: Krawtchouk bound 1 n-1}
If $j \neq \frac{n}{2}$, then for all $i \in [n-1]$,
\[ |K_j(1)|=\left|\binom{n}{j}\frac{n-2j}{n}\right| \geq |K_j(i)|. \]
\end{lemma}

\section{Uniqueness of optima}\label{section: uniqueness of optima}
\subsection{The upper half case}
In this section, we prove that the Delsarte linear program has a unique optimum when $2\ceil{d/2}>n-d$, roughly corresponding to $d$ being at least $n/2$: precisely, if $d$ is even, then $2d>n$, and when $d$ is odd, then $2d\geq n$.
\begin{theorem}\label{theorem: unique optimum upper half}
If $2\ceil{d/2}>n-d$, the Delsarte linear program has a unique optimum, namely the quasicode $A^*$ given by
\[ A^*_i = \begin{cases} 1 & \text{if } i=0 \\ \frac{n}{2d-n} & \text{if } i=d \\ 0 & \text{otherwise} \end{cases} \]
when $d$ is even, and
\[ A^*_i = \begin{cases} 1 & \text{if } i=0 \\ \frac{d+1}{2d-n+1} & \text{if } i=d \\ \frac{n-d}{2d-n+1} & \text{if } i = d+1 \\ 0 & \text{otherwise} \end{cases} \]
when $d$ is odd.
\end{theorem}
\begin{proof}
We first address the even $d$ case.
The objective value of $A^*$ is $\frac{2d}{2d-n}$, and all of the constraints trivially hold except for the Delsarte inequalities, which become
\[ \sum_{i=0}^n A_i^* K_j(i) = \binom{n}{j}+\frac{n}{2d-n}K_j(d) \geq 0 \]
for all $j \in [n]$, or equivalently
\[ K_j(d) \geq \frac{n-2d}{n}\binom{n}{j}. \]
Applying reciprocity of the Krawtchouk polynomials yields
\[ K_d(j) \geq \frac{n-2d}{n}\binom{n}{d}. \]
As $2d>n$, \cref{lemma: Krawtchouk bound 1 n-1} implies this inequality for all $j \in [n-1]$.
For $j=n$,
\[ K_d(n)=(-1)^d K_d(0) = \binom{n}{d} > 0 > \frac{n-2d}{n}\binom{n}{d},\]
which completes the proof that $A^*$ is feasible.

Consider the dual solution $c^*$ given by
\[ c^*_j = \begin{cases} 1 & \text{if } j = 0 \\ \frac{2d-n+1}{(2d-n)(2d-n+2)} & \text{if } j = 1 \\ \frac{1}{(2d-n)(2d-n+2)} & \text{if } j = n-1 \\ 0 & \text{otherwise}. \end{cases} \]
It has dual objective of $\frac{2d}{2d-n}$, so demonstrating it is a feasible solution proves $A^*$ is optimal.
As $2d>n$, we find that $c_j^* \geq 0$ for all $j \in [n]$, so it remains to show that for all $i \in [d,n]$,
\[ \sum_{j=0}^n c_j^* K_j(i) \leq 0. \]
Note that $K_0(i)=1$, $K_1(i)=n-2i$, and $K_{n-1}(i)=(-1)^i(n-2i)$, yielding
\[ \sum_{j=0}^n c_j^* K_j(i) = 1 + (n-2i)\frac{2d-n+1}{(2d-n)(2d-n+2)} + (-1)^i (n-2i)\frac{1}{(2d-n)(2d-n+2)}. \]
When $i$ is even, as $i \geq d$ this becomes
\[ 1 + \frac{n-2i}{2d-n} \leq 1+\frac{n-2d}{2d-n} = 0, \]
and when $i$ is odd, we have $i \geq d+1$ as $d$ is even, so the RHS becomes
\[ 1+\frac{n-2i}{2d-n+2} \leq 1+\frac{n-2d-2}{2d+2-n} = 0, \]
as desired.
Notice that these constraints are tight if and only if $i=d$ or $d+1$.

Hence, $A^*$ and $c^*$ are optimal solutions to the primal and dual, respectively.
By complementary slackness, as the dual constraint is strict for $i > d+1$, any optimum to the primal must satisfy $A_i = 0$ for all $i >d+1$.
Hence, $A_0=1$, $A_d$, and $A_{d+1}$ are the only potentially nonzero values.
As $c_1^*,c_{n-1}^* > 0$ then the Delsarte inequality must be tight for $j=1$ and $n-1$ for any primal optimum.
These conditions become
\begin{align*}
    n + (n-2d) A_d + (n-2d-2) A_{d+1} &= 0 \\
    n + (n-2d) A_d - (n-2d-2) A_{d+1} &= 0,
\end{align*}
which requires $A_{d+1}=0$ and $A_d=\frac{n}{2d-n}$, proving that $A^*$ is the unique optimum.

We now address the odd $d$ case, where $2d\geq n$.
The objective value of $A^*$ is $\frac{2d+2}{2d-n+1}$, and all of the constraints trivially hold except for the Delsarte inequalities, which become
\[ \sum_{i=0}^n A_i^* K_j(i) = \binom{n}{j}+\frac{d+1}{2d-n+1}K_j(d)+\frac{n-d}{2d-n+1}K_j(d+1) \geq 0 \]
for all $j \in [n]$.
Multiplying by $\frac{n!}{(d+1)!(n-d)!}$ yields the equivalent inequality
\[ \binom{n}{d} \frac{K_j(d)}{2d-n+1} + \binom{n}{d+1}\frac{K_j(d+1)}{2d-n+1} \geq -\frac{n!}{(d+1)!(n-d)!}\binom{n}{j}, \]
and after applying reciprocity of Krawtchouk polynomials, this becomes
\[ \frac{K_d(j)}{2d-n+1}+\frac{K_{d+1}(j)}{2d-n+1} \geq -\frac{n!}{(d+1)!(n-d)!}. \]
By \cref{lemma: Krawtchouk recurrence}, this is equivalent to
\[ K_{d+1}(j;n+1) \geq -\frac{n!(2d-n+1)}{(d+1)!(n-d)!} = \binom{n+1}{d+1}\frac{(n+1)-2(d+1)}{n+1}.\]
As $d+1>\frac{n+1}{2}$, \cref{lemma: Krawtchouk bound 1 n-1} implies this inequality holds for all $j\in[n]$.
Hence, $A^*$ is a feasible solution.

Consider the dual solution $c^*$ given by
\[ c_j^* = \begin{cases} 1 & \text{if } j = 0 \\ \frac{1}{2d-n+1} & \text{if } j\in\{1,n\} \\ 0 & \text{otherwise}. \end{cases} \]
It has dual objective of $\frac{2d+2}{2d-n+1}$, so showing $c^*$ is a feasible solution proves $A^*$ is optimal.
The assumption that $2d \geq n$ ensures $c_j^* \geq 0$ for all $j$, and so it remains to show that for all $i \in [d,n]$,
\[ \sum_{j=0}^n c_j^* K_j(i) \leq 0. \]
We have
\[ \sum_{j=0}^n c_j^* K_j(i) = 1 + \frac{n-2i}{2d-n+1} + (-1)^i \frac{1}{2d-n+1}. \]
When $i$ is even, $i \geq d+1$ so this becomes
\[ 1 + \frac{n+1-2i}{2d-n+1} \leq 1+\frac{n-2d-1}{2d-n+1} = 0, \]
and when $i$ is odd, as $i \geq d$ we find
\[ 1 + \frac{n-1-2i}{2d-n+1} \leq 1+\frac{n-2d-1}{2d-n+1} = 0, \]
as desired.
These constraints are tight if and only if $i=d$ or $d+1$.

Thus, $A^*$ and $c^*$ are optimal solutions to the primal and dual, respectively.
By complementary slackness, any primal optimum must additionally have $A_i=0$ for all $i>d+1$.
As $c_1^*,c_n^*>0$ then the Delsarte inequality must be tight for $j=1$ and $n$ for any primal optimum, i.e.,
\begin{align*}
    n + (n-2d) A_d + (n-2d-2) A_{d+1} &= 0 \\
    1 - A_d + A_{d+1} &= 0.
\end{align*}
Substituting the latter into the former yields
\begin{align*}
    n + (n-2d)(A_{d+1}+1)+(n-2d-2)A_{d+1} = 2n-2d+(2n-4d-2)A_{d+1} = 0,
\end{align*}
whose solution is $A_{d+1}=\frac{n-d}{2d-n+1}$.
Requiring $A_d=A_{d+1}+1$ yields $A^*$ as the unique optimum.
\end{proof}

\begin{example}
For an example of a code realizing the optimum from \cref{theorem: unique optimum upper half}, the binary simplex code \cite[Ch. 1, \S 9]{macwilliams1977theory}, which is the dual of a Hamming code, has $n=2^r-1$ and $d=2^{r-1}$, where the code is supported at exactly this one distance, i.e., $S=\{d\}$.
It has size $2^r$, and thus has distance distribution given by $A_0=1$, $A_d=2^r-1$, and $A_i=0$ everywhere else.
This matches the even case of \cref{theorem: unique optimum upper half}, and puncturing this code, i.e., removing a bit from the code, corresponds to the odd case of \cref{theorem: unique optimum upper half}.
\end{example}

\begin{remark}\label{remark:plotkin}
The case $2d>n$ is known as the Plotkin range.
For codes, which must have integer size, the Plotkin bound (see \cite[p.\ 43]{macwilliams1977theory}) upper bounds the size of codes by $2\left\lfloor\frac{d}{2d-n}\right\rfloor$ when $d$ is even, and $2\left\lfloor\frac{d+1}{2d-n+1}\right\rfloor$ when $d$ is odd.
These are the same bounds as for quasicodes from \cref{theorem: unique optimum upper half}, except rounding down to the nearest even integer.
Provided enough Hadamard matrices exist, the Plotkin bound is tight (see \cite[p.\ 50]{macwilliams1977theory}).

Delsarte \cite[Theorem~14]{delsarte1972bounds} provided an upper bound of $\frac{2d}{2d-n}$ on the optimal objective value of the Delsarte linear program in the Plotkin range.
\cref{theorem: unique optimum upper half} proves this bound is tight when $d$ is even, and improves the bound to $\frac{2d+2}{2d-n+1}<\frac{2d}{2d-n}$ when $d$ is odd, which is tight for the optimal objective value of the Delsarte linear program.
\end{remark}

\subsection{Non-uniqueness of dual optimum}\label{subsection: non-unique dual}

Studying the cases $d=n$ and $n-1$, we find that the dual of the Delsarte linear program does not generally have a unique optimum, i.e., a unique feasible point achieving the optimal value, for all pairs $(n,d)$.

\begin{proposition}\label{proposition: non-unique optimal dual d=n}
When $d=n$, the dual of the Delsarte linear program has a unique optimum if and only if $n \leq 2$.
\end{proposition}
\begin{proof}
By complementary slackness with primal optimum $A^*=(1,0,0,\dots,0,1)$ from \cref{theorem: unique optimum upper half}, a dual optimum must have
\begin{equation}\label{eq: dual d=n complementary slackness}
    \sum_{j=0}^n c_j K_j(n) = 1+\sum_{j=1}^n c_j (-1)^j \binom{n}{j} = 0,
\end{equation}
using properties from \cref{lemma: Krawtchouk basic properties}.
As $c_j \geq 0$ for all $j$, clearly $c_j=0$ for any positive even $j$, as otherwise we can decrease it along with some $c_j'>0$ for odd $j'$ to decrease the objective value while preserving the necessary equation, \cref{eq: dual d=n complementary slackness}.
Thus, any dual optimum must have
\begin{equation}\label{eq: dual d=n necessary and sufficient}
    \sum_{\substack{1 \leq j \leq n \\ j \text{ odd}}} c_j \binom{n}{j} = 1,
\end{equation}
which implies the objective value is 2, which is optimal, so this condition, along with $c_j=0$ for all positive even $j$, is necessary and sufficient for a dual optimum.

For $n \leq 2$ this clearly yields a unique dual optimum, but for $n \geq 3$, we find \cref{eq: dual d=n necessary and sufficient} is a equation over multiple variables, yielding multiple optimal solutions.
\end{proof}

\begin{theorem}\label{theorem: non-unique odd optimal dual d=n-1}
When $d=n-1$, the dual of the Delsarte linear program has a unique optimum if and only if $n$ is even, in which case the unique dual optimum is $c^*=\left(1,\frac{1}{n-1},0,0,\dots,0,\frac{1}{n-1}\right)$.
\end{theorem}
\begin{proof}
As $d=n-1 \geq 1$, we have $n \geq 2$.
From the proof of \cref{theorem: unique optimum upper half}, we find that, with respect to the unique optimum $A^*$, the constraint $\sum_{i=0}^n A^*_i K_j(i) \geq 0$ is strict for all $j \in [n]$ except $j=1$ and $j=2\floor{n/2}$.
By complementary slackness, any optimal dual solution $c$ can thus only have positive entries $c_0=1$, $c_1$, and $c_{2\floor{n/2}}$.
Complementary slackness when $A^*_i > 0$ also yields $\sum_{j=0}^n c_j K_j(i) = 0$ for $i\in\{n-1,n\}$ if $n$ is even, and for $i=n-1$ if $n$ is odd.

For even $n \geq 2$, implementing these observations from complementary slackness allows the dual to be rewritten as
\begin{align*}
    \min \quad & 1 + nc_1 + c_n \\
    \text{such that} \quad & 1 - (n-2) c_1 - c_n = 0 \\
    & 1 - n c_1 + c_n = 0 \\
    & c_1,c_n \geq 0.
\end{align*}
The first two constraints have a unique solution $c_1=c_n=\frac{1}{n-1}$, meaning the unique dual optimum is $c^*=\left(1,\frac{1}{n-1},0,0,\dots,0,\frac{1}{n-1}\right)$.

For odd $n \geq 3$, we rewrite the dual as
\begin{align*}
    \min \quad & 1 + nc_1 + n c_{n-1} \\
    \text{such that} \quad & 1 - (n-2) c_1 - (n-2) c_{n-1} = 0 \\
    & 1 - nc_1 + nc_{n-1} \leq 0 \\
    & c_1, c_{n-1} \geq 0.
\end{align*}
The first constraint yields $c_1 + c_{n-1} = \frac{1}{n-2}$, which fixes the objective value to be $1+\frac{n}{n-2}$, which is the optimal value by \cref{theorem: unique optimum upper half}.
So any feasible solution is a dual optimum, and namely substituting $c_{n-1}=\frac{1}{n-2}-c_1$, the second constraint becomes $c_1 \geq \frac{n-1}{n(n-2)}$.
Combining this with the nonnegativity constraint $\frac{1}{n-2}-c_1=c_{n-1} \geq 0$ yields
\[ \frac{n-1}{n}\cdot\frac{1}{n-2} \leq c_1 \leq \frac{1}{n-2}, \]
so there are multiple values for $c_1$, all of which yield a nonnegative value for $c_{n-1}$.
Each of these yields a feasible dual solution, which must then be a dual optimum, so the dual does not have a unique optimum when $n$ is odd.
\end{proof}

\subsection{The case $d=1$}

When $d=1$, we see that both the primal and dual have unique optima.

\begin{theorem}\label{theorem: unique optimal quasicode and dual d=1}
When $d=1$, the Delsarte linear program has a unique optimum, namely the quasicode $A^*$ given by $A^*_i = \binom{n}{i}$ for all $0 \leq i \leq n$.
In addition, the dual linear program has a unique optimum, namely the dual solution $c^*=\1$.
\end{theorem}
The fact that $A_i^*=\binom{n}{i}$ is optimal is well-known; see Levenshtein \cite{levenshtein1992designs}.
For completeness, we prove that it is not only optimal, but also unique, and additionally address the dual linear program.
\begin{proof}
Define vectors $K_j'=(K_j(1),\dots,K_j(n))$ for all $1 \leq j \leq n$.
We first show that these vectors are linearly independent.
Suppose we have $b_1,\dots,b_n$ such that $b_1K_1'+\cdots+b_nK_n' = \0$.
Define vectors $K_j=(K_j(0),\dots,K_j(n))$ for all $1 \leq j \leq n$, and let $v$ be the vector
\[ v=b_1K_1+\cdots+b_nK_n=\left(b_1\binom{n}{1}+\cdots+b_n\binom{n}{n},0,\dots,0\right). \]
By the orthogonality of the Krawtchouk polynomials, for all $1 \leq j \leq n$,
\[ \binom{n}{j}\left(b_1\binom{n}{1}+\cdots+b_n\binom{n}{n}\right) = \sum_{i=0}^n \binom{n}{i} v_i K_j(i) = 2^n \binom{n}{j} b_j, \]
so
\[ b_1\binom{n}{1}+\cdots+b_n\binom{n}{n}=2^n b_j \]
for all $j$.
Hence $b_1=\cdots=b_n$, and the necessary equation is
\[ b_1 \left(\binom{n}{1}+\cdots+\binom{n}{n}\right) = b_1\left(2^n-1\right) = 2^n b_1,\]
which yields $b_1=\cdots=b_n=0$, implying the vectors $K_j'$ for $1 \leq j \leq n$ are linearly independent.

We first show that $A^*$ is the unique optimum of the primal linear program.
The linear independence of the vectors $K_j'$ implies that there is at most one vertex $A$ of $P$ that has $A_i > 0$ for all $i \in [n]$, namely the unique solution to $\sum_{i=0}^n A_i K_j(i) = 0$ for all $j \in [n]$ where $A_0=1$.
This point is $A^*$, as $\sum_{i=0}^n \binom{n}{i} K_j(i)=0$ for all $j\in[n]$ via orthogonality of the Krawtchouk polynomials, using $K_0(i)=1$ as the other polynomial.
Optimality of $A^*$ can be shown by finding a dual solution that achieves the same objective value, namely $2^n$.
The dual solution $c^*=\1$ is such a solution, as it yields an objective value of $2^n$ and is feasible because for any $i \in [n]$,
\[ \sum_{j=0}^n K_j(i) = 0 \]
by \cref{lemma: Krawtchouk column-sum zero}.

By complementary slackness with this optimal dual solution $c^*=\1$, this means the Delsarte inequalities $\sum_{i=0}^n A_i K_j(i) \geq 0$ must be sharp for all $j \in [n]$.
We know that there is only one solution to this system of equations, which is $A^*$, so $A^*$ is the unique optimal quasicode.

To show that $c^*=\1$ is the unique optimum of the dual linear program, by complementary slackness with optimal primal solution $A^*$ where $A^*_i=\binom{n}{i}>0$ for all $i$, the dual inequalities $\sum_{j=0}^n c_j K_j(i) \leq 0$ for all $i \in [n]$ must be sharp.
Equivalently, we must have $\sum_{j=1}^n c_j K_j(i) = -1$ for all $i \in [n]$.
The linear independence of vectors $K_j'$ implies that the $n\times n$ matrix $K'$ given by $K'_{ij}=K_i(j)$ is non-singular, so its columns, which are the vectors $(K_1(i),\dots,K_n(i))$ for $1 \leq i \leq n$, are linearly independent.
Thus, the system of equations $\sum_{j=1}^n c_j K_j(i) = -1$ for all $i \in [n]$ has a unique solution, which we have already shown is given by $c_1=\cdots=c_n=1$.
Therefore, $c^*=\1$ is the unique optimum of the dual.
\end{proof}

\subsection{The case $d=2$}

In this section we will show that although the primal always has a unique optimum when $d=2$, the dual almost always does not.

\begin{theorem}\label{theorem: unique optimal quasicode d=2}
When $d=2$, the Delsarte linear program has a unique optimum, namely the quasicode $A^*$ given by $A^*_i = \begin{cases} \binom{n}{i} & i \text{ is even} \\ 0 & i \text{ is odd}\end{cases}$ for all $0 \leq i \leq n$.
\end{theorem}
Similar to the case $d=1$, it is already known that the stated $A^*$ is optimal; see Levenshtein \cite{levenshtein1992designs}.
However, the uniqueness of this optimum, as well as the dual as addressed in \cref{theorem: non-unique optimal dual d=2}, have not been demonstrated previously, so for completeness we provide a full proof of the result.
\begin{proof}
Notice that the objective value for $A^*$ is $2^{n-1}$.
We claim that $c_j=\frac{n-j}{n}$ for all $0 \leq j \leq n$ is a dual solution with objective value $2^{n-1}$, which proves $A^*$ is optimal.

To show $c$ is a feasible solution to the dual, we will show that for all $i \in [2,n]$,
\[ \sum_{j=0}^n c_j K_j(i) = 0.\]
Notice that
\[ \sum_{j=0}^n c_j K_j(i) = \sum_{j=0}^n K_j(i) - \frac{1}{n}\sum_{j=0}^n j K_j(i) = -\frac{1}{n}\sum_{j=0}^n j K_j(i)\]
by \cref{lemma: Krawtchouk column-sum zero}, so it is equivalent to show
\[ \sum_{j=0}^n j K_j(i) = 0.\]
Directly from the definition of $K_j(i)$, we see that $K_0(i)=1$ and $K_1(i)=n-2i$, meaning
\[ i=\frac{n}{2}K_0(i)-\frac{1}{2}K_1(i).\]
Using \cref{lemma: Krawtchouk basic properties} and then \cref{lemma: Krawtchouk orthogonality}, we find
\[ \sum_{j=0}^n j K_j(i) = \frac{1}{\binom{n}{i}} \sum_{j=0}^n \binom{n}{j} j K_i(j) = \frac{1}{\binom{n}{i}} \sum_{j=0}^n \binom{n}{j} \left(\frac{n}{2}K_0(j)-\frac{1}{2}K_1(j)\right)K_i(j) = 0\]
for all $i \in [2,n]$, as desired.

The objective value of $c$ is
\[ \sum_{j=0}^n c_j \binom{n}{j} = \sum_{j=0}^n \frac{n-j}{n} \binom{n}{j} = 2^n - \frac{1}{n}\sum_{j=1}^n j\binom{n}{j} = 2^n \frac{1}{n}n\cdot2^{n-1} = 2^{n-1},\]
where the combinatorial identity
\[ \sum_{j=1}^n j\binom{n}{j} = n 2^{n-1} \]
can be easily seen by noticing both sides count the ways to choose a subset of a set of $n$ elements, where one element of the subset is distinguished: the left hand side first picks the subset of $j$ elements and then picks the distinguished element, while the right hand side picks the distinguished element and then considers whether each of the remaining $n-1$ elements are in the subset.

By complementary slackness, as $c_j > 0$ for all $j \in [n-1]$, any optimal quasicode $A$ must satisfy
\[ \sum_{i=0}^n A_i K_j(i) = 0\]
for all $j \in [n-1]$.
Note that $A^*$ satisfies these conditions, because expressing $A^*$ by
\[ A^*_i = \binom{n}{i}(K_0(i)+K_n(i))/2 \]
yields
\[ \sum_{i=0}^n A_i K_j(i) = \frac{1}{2} \sum_{i=0}^n \binom{n}{i} (K_0(i)+K_n(i))K_j(i) = 0\]
by \cref{lemma: Krawtchouk orthogonality} for all $j \in [n-1]$.
We have $n-1$ linear equations, one for each $j\in[n-1]$, that must be satisfied for any optimal quasicode $A$; given that $A_1=0$, we also have $n-1$ variables, $A_2,\dots,A_n$.
Let $K_j^{(2)}=\left(K_j(2),\dots,K_j(n)\right)$ for $j \in [n-1]$.
If the vectors $K_j^{(2)}$ for $j \in [n-1]$ are linearly independent, then this implies there is a unique solution to our system of $n-1$ linear equations, and thus $A^*$ is the unique optimal quasicode.

Suppose $b_1 K_1^{(2)}+\cdots+b_{n-1}K_{n-1}^{(2)} = \0$.
Let $v=b_1K_1+\cdots+b_{n-1} K_{n-1}$, using the same definition $K_j=(K_j(0),\dots,K_j(n))$ as in \cref{theorem: unique optimal quasicode and dual d=1}.
Using orthogonality of the Krawtchouk polynomials with $K_0(i)=1$, we find
\[ 0 = \sum_{i=0}^n \binom{n}{i} v_i = v_0 + n v_1,\]
and using $K_n(i)=(-1)^i$, we find
\[ 0 = \sum_{i=0}^n \binom{n}{i} v_i (-1)^i = v_0 - n v_1.\]
This implies $v_0=v_1=0$, so $v=\0$.
Hence, for all $j \in [n-1]$, we find
\[ 0 = \sum_{i=0}^n \binom{n}{i} v_i K_j(i) = 2^n \binom{n}{j} b_j,\]
so $b_1=\cdots=b_{n-1}=0$, and the vectors $K_j^{(2)}$ are linearly independent.

This completes the proof that $A^*$ is the unique optimum.
\end{proof}
By contrast, the dual almost never has a unique optimum.
\begin{theorem}\label{theorem: non-unique optimal dual d=2}
When $d=2$, the dual of the Delsarte linear program has a unique optimum if and only if $n = 2$.
\end{theorem}
\begin{proof}
\cref{proposition: non-unique optimal dual d=n} implies the result when $n=d=2$, so it suffices to show that for all $n \geq 3$, the dual of the Delsarte linear program with $d=2$ does not have a unique optimal solution.

As the objective value was demonstrated in \cref{theorem: unique optimal quasicode d=2} to be $2^{n-1}$, a dual solution $c$ is an optimum if and only if the following properties are satisfied:
\begin{align}
    \sum_{j=0}^n c_j \binom{n}{j} &= 2^{n-1} \label{eq: d=2 dual condition 1}\\
    \sum_{j=0}^n c_j K_j(i) &\leq 0 \quad \text{for all } i \in [2,n] \label{eq: d=2 dual condition 2}\\
    c_j &\geq 0 \quad \text{for all } j \in [n] \label{eq: d=2 dual condition 3}\\
    c_0 &= 1. \label{eq: d=2 dual condition 4}
\end{align}
Expressing $c$ as a linear combination of the Krawtchouk polynomials, i.e., for all $0 \leq j \leq n$,
\[ c_j = b_0 K_0(j) + \cdots + b_n K_n(j), \]
then notice that \cref{eq: d=2 dual condition 1} is equivalent to
\begin{align*}
    \frac{1}{2} = \frac{1}{2^n}\sum_{j=0}^n \binom{n}{j} c_j = \frac{1}{2^n}\sum_{j=0}^n \binom{n}{j} c_j K_0(j) = b_0,
\end{align*}
using the orthogonality of the Krawtchouk polynomials.
By first applying reciprocity and then orthogonality of the Krawtchouk polynomials, we find
\[ \sum_{j=0}^n c_j K_j(i) = \frac{1}{\binom{n}{i}}\sum_{j=0}^n \binom{n}{j}c_j K_i(j) = 2^n b_i.\]
Therefore, \cref{eq: d=2 dual condition 2} is equivalent to $b_i \leq 0$ for all $i \in [2,n]$.

If $n\geq 3$ is odd, we claim the dual solution $c$ given by
\[ c_j = \frac{1}{2}K_0(j) + \frac{1+\varepsilon}{2n}K_1(j)-\frac{\varepsilon}{2}K_n(j) = \frac{1}{2}+\frac{1+\varepsilon}{2}-\frac{1+\varepsilon}{n}j-(-1)^j\frac{\varepsilon}{2} = \begin{cases} 1-\frac{1+\varepsilon}{n}j & j \text{ is even} \\ (1+\varepsilon)\frac{n-j}{n} & j \text{ is odd} \end{cases}\]
for any $0 \leq \varepsilon \leq \frac{1}{n-1}$ is an optimal solution.
As $b_0=\frac{1}{2}$ in this case and $b_i=0$ for all $i \in [2,n-1]$ and $b_n=-\varepsilon/2\leq 0$, we find \cref{eq: d=2 dual condition 1,eq: d=2 dual condition 2} hold.
We directly verify $c_0=1$, and thus it remains to show $c_j \geq 0$ for all $j \in [n]$.
If $j$ is odd, as $0 \leq j \leq n$, we find $c_j \geq 0$.
If $j$ is even, then as $n$ is odd, we have $j \leq n-1$, and thus
\[ c_j = 1-\frac{1+\varepsilon}{n}j \geq 1-(1+\varepsilon)\frac{n-1}{n} = \frac{1-(n-1)\varepsilon}{n} \geq 0,\]
as $\varepsilon \leq \frac{1}{n-1}$.
Hence any $0\leq\varepsilon\leq\frac{1}{n-1}$ yields a distinct optimum $c$, so for odd $n \geq 3$, there are multiple dual optima.

If $n \geq 4$ is even, then we claim the dual solution $c$ given by
\[ c_j = \frac{1}{2}K_0(j) + \frac{1+\varepsilon}{2n}K_1(j)-\frac{\varepsilon}{2n}K_{n-1}(j) = \begin{cases} \frac{n-j}{n} & j \text{ is even} \\ 1+\varepsilon-\frac{1+2\varepsilon}{n}j & j \text{ is odd} \end{cases}\]
for any $0 \leq \varepsilon \leq \frac{1}{n-2}$ is an optimal solution.
As $b_0=\frac{1}{2}$ in this case and $b_i\leq 0$ for all $i \in [2,n]$, we find \cref{eq: d=2 dual condition 1,eq: d=2 dual condition 2} hold.
We directly verify $c_0 = 1$, and thus it remains to show $c_j \geq 0$ for all $j \in [n]$.
If $j$ is even, as $0 \leq j \leq n$, we find $c_j \geq 0$.
If $j$ is odd, then as $n$ is even, we have $j \leq n-1$, and thus
\[ c_j = 1 + \varepsilon - \frac{1+2\varepsilon}{n}j \geq 1+\varepsilon -(1+2\varepsilon)\frac{n-1}{n} = \frac{1-(n-2)\varepsilon}{n} \geq 0,\]
as $\varepsilon\leq\frac{1}{n-2}$.
Hence any $0 \leq \varepsilon \leq \frac{1}{n-2}$ yields a distinct optimum $c$, so for even $n \geq 4$, there are multiple dual optima.
This completes the proof.
\end{proof}

\subsection{Non-uniqueness of primal optimum}\label{subsection: non-unique primal}
While our previous results for $d=1,2$, and for the upper half $2\ceil{d/2}+d>n$ have all had a unique primal optimum, this uniqueness does not hold in general.
The smallest $n$ for which the Delsarte linear program does not have a unique optimum is $(n,d)=(17,5)$, and the next smallest cases are $(21,5)$ and $(23,5)$.
For example, the optimal solutions for $(17,5)$ are the points on the line segment between the two quasicodes
\[ \left(1, 0, 0, 0, 0, 52, \frac{304}{3}, \frac{176}{3}, \frac{250}{3}, \frac{520}{3}, \frac{368}{3}, \frac{112}{3}, 32, 20, 0, 0, 1, 0\right) \]
and
\[ \left(1, 0, 0, 0, 0, 51, \frac{307}{3}, \frac{191}{3}, \frac{235}{3}, \frac{490}{3}, \frac{398}{3}, \frac{142}{3}, 22, 15, 5, 1, 0, 0\right). \]
These three cases were found via a computer search solving a system of linear inequalities and equations using the optimal objective value, Delsarte inequalities, and complementary slackness conditions.
No other cases exist for $1 \leq d \leq n \leq 23$.
This prompts the question of when the Delsarte linear program has a unique optimum.

\begin{question}\label{question: uniqueness}
For which values of $(n,d)$ does the Delsarte linear program have a unique optimum?
Similarly, for which values of $(n,d)$ does the dual have a unique optimum?
\end{question}

\section{Krawtchouk decomposition of optimal quasicodes}\label{section: Krawtchouk decomposition of optimum}
In this section, we further characterize optimal quasicodes by first applying a transformation to the quasicodes.
Suppose $A$ is a quasicode of the Delsarte linear program for a given pair of values $(n,d)$.
Let $A'$ be the vector given by $A_i'=A_i\cdot 2^n/\binom{n}{i}$ for $0 \leq i \leq n$.
If $A$ is the distance distribution of a code $\mathcal{C}$, then $A'_i$ is $4^n/|\mathcal{C}|$ times the proportion of ordered pairs $(x,y)$ of words $x,y\in\F_2^n$ a Hamming distance $i$ apart whose two elements $x$ and $y$ are both in $\mathcal{C}$.
Then there is a unique vector $b$ such that $A'=b K$, given by $b = \frac{1}{2^n}A' K$, or equivalently
\[ b_j = \sum_{i=0}^n \frac{A_i K_i(j)}{\binom{n}{i}} = \frac{1}{\binom{n}{j}}\sum_{i=0}^n A_i K_j(i). \]
This means for all $i$,
\[ A_i = \binom{n}{i} (b_0 K_0(i) + \cdots + b_n K_n(i)) / 2^n. \]
We will refer to this vector $b$ as the \emph{Krawtchouk decomposition} of the quasicode $A$.
Then the Delsarte inequalities for $A$ can be rewritten as $\binom{n}{j} b_j \geq 0$ for all $j \in [n]$, or equivalently simply $b_j \geq 0$.
And the objective function simplifies to $b_0$:
\begin{align*}
    \sum_{i=0}^n A_i
    &= \frac{1}{2^n} \sum_{i=0}^n \binom{n}{i} \sum_{j=0}^n b_j K_j(i)
    = \frac{1}{2^n} \sum_{j=0}^n b_j \sum_{i=0}^n \binom{n}{i} K_j(i) K_0(i)
    = \sum_{j=0}^n b_j \binom{n}{j} \delta_{j0} = b_0.
\end{align*}

Hence, the Delsarte linear program can be rephrased using the Krawtchouk decomposition as follows:
\begin{align*}
    \max \quad & b_0 \\
    \text{such that} \quad & \sum_{j=0}^n \binom{n}{j} b_j = 2^n \\
    & \sum_{j=0}^n b_j K_j(i) \geq 0 \quad \text{for all } i \in [n] \\
    & \sum_{j=0}^n b_j K_j(i) = 0 \quad \text{for all } i \in [d-1] \\
    & b_j \geq 0 \quad \text{for all } j \in [n].
\end{align*}

\cref{theorem: unique optimal quasicode and dual d=1} shows that the unique optimal quasicode $A^*$ when $d=1$ has Krawtchouk decomposition $b=(2^n,0,0,\dots,0)$.
Similarly, \cref{theorem: unique optimal quasicode d=2} shows that when $d=2$, the optimum $A^*$ has Krawtchouk decomposition $b=(2^{n-1},0,0,\dots,0,2^{n-1})$.

The Krawtchouk decomposition of the unique optimum for the upper half is given in the following result.
But first, we will introduce some notation to simplify our future discussions.
Define $h=n-d$, and let $k=2\ceil{d/2}$ be the smallest even integer at least $d$.
Using this notation, the upper half condition simply becomes $k>h$.

\begin{theorem}\label{theorem: upper half Krawtchouk decomposition}
For all $d \geq 1$ and all $h \geq 0$ such that $k=2\ceil{d/2}>h$, the unique optimum to the Delsarte linear program with $(n,d)=(d+h,d)$ has Krawtchouk decomposition $b^*$ given by
\[ b_j^* = 1 - \frac{K_k(j;h+k)}{K_k(1;h+k)}. \]
\end{theorem}
\begin{proof}
If $d$ is even, then our expression for $b^*$ becomes
\[ b_j^* = 1 - \frac{K_d(j;n)}{K_d(1;n)}. \]
Then the corresponding quasicode is given by
\[ A_i = \binom{n}{i}\sum_{j=0}^n b_j^* K_j(i;n)/2^n. \]
By \cref{lemma: Krawtchouk column-sum zero}, the constant 1 term in $b^*$ contributes 1 to $A_0$ and 0 everywhere else.
Thus we expand $b_j^*$ to yield
\begin{align*}
    A_i
    &= \delta_{i0} - \binom{n}{i}\sum_{j=0}^n \frac{K_d(j;n)}{K_d(1;n)}K_j(i;n)/2^n
    = \delta_{i0} - \sum_{j=0}^n \frac{K_d(j;n)}{K_d(1;n)}\binom{n}{j}K_i(j;n)/2^n
    \\ &= \delta_{i0} - \binom{n}{d}\delta_{id}\frac{1}{K_d(1;n)}
    = \delta_{i0} + \delta_{id} \frac{n}{2d-n},
\end{align*}
using reciprocity and orthogonality of the Krawtchouk polynomials.
But this expression equals the unique primal optimum $A^*$ as given by \cref{theorem: unique optimum upper half}.

If $d$ is odd, then our expression for $b^*$ becomes
\[ b_j^* = 1 - \frac{K_{d+1}(j;n+1)}{K_{d+1}(1;n+1)}, \]
and the corresponding quasicode is given by
\begin{align*}
    A_i
    &= \delta_{i0} - \sum_{j=0}^n \frac{K_{d+1}(j;n+1)}{K_{d+1}(1;n+1)} \binom{n}{j}K_i(j;n)/2^n
    = \delta_{i0} - \sum_{j=0}^n \frac{K_{d+1}(j;n)+K_d(j;n)}{K_{d+1}(1;n+1)} \binom{n}{j}K_i(j;n)/2^n
    \\ &= \delta_{i0} - \binom{n}{d+1}\frac{\delta_{i,d+1}}{K_{d+1}(1;n+1)} - \binom{n}{d}\frac{\delta_{id}}{K_{d+1}(1;n+1)}
    \\ &= \delta_{i0} + \delta_{id}\frac{d+1}{2d-n+1} + \delta_{i,d+1}\frac{n-d}{2d-n+1},
\end{align*}
where we use \cref{lemma: Krawtchouk recurrence} to reduce $K_{d+1}(j;n+1)$ to block length $n$.
This expression equals the unique primal optimum $A^*$ as given by \cref{theorem: unique optimum upper half}, completing the proof.
\end{proof}
The presence of $k=2\ceil{d/2}$, which rounds $d$ up to the nearest even integer, suggests that there are some parity connections between neighboring values of $(n,d)$.
In particular, if $d=2e$ is even, then $(n-1,2e-1)$ and $(n,2e)$ have the same values of $k$ and $h$, and thus their unique optima $b^*$ of the Krawtchouk decomposition LP are the same, up to truncating the $n$th entry $b^*_n$ for the $(n-1,2e-1)$ case.
This parity phenomenon holds for all $1 \leq d \leq n$, not just the upper half, as seen in the following result.

\begin{theorem}\label{theorem: parity}
For all positive integers $n$ and $e$ where $2e \leq n$, there exist optima $A^E$ and $A^O$ of the $(n,2e)$ and $(n-1,2e-1)$ Delsarte linear programs, respectively, whose Krawtchouk decompositions agree on indices 0 through $n-1$, inclusive.
\end{theorem}

We need additional tools, developed in \cref{section: extending and puncturing}, in order to prove \cref{theorem: parity}, so we defer the proof to \cref{section: extending and puncturing}.

We also have another result on the symmetry of the optimal Krawtchouk decompositions that is closely related to the aforementioned parity phenomenon.
Again, we defer the proof to \cref{section: extending and puncturing}.

\begin{theorem}\label{theorem: symmetry}
For all $1 \leq d \leq n$ where $d$ is even, the Delsarte linear program for $(n,d)$ has an optimum $A^E$ whose Krawtchouk decomposition $b^E$ satisfies $b_j^E=b_{n-j}^E$ for all $0 \leq j \leq n$.
\end{theorem}

We first comment that this symmetry property is equivalent to the quasicodes being \emph{even}, i.e., having a support contained in the set of even integers.
From this interpretation, it is intuitively clear why $d$ must be even.
If the symmetry property holds, then for odd $i$, we find
\[ A_i = \frac{\binom{n}{i}}{2^n}\sum_{j=0}^n b_j K_j(i) = \frac{\binom{n}{i}}{2^n}\sum_{j=0}^n b_{n-j} (-1)^i K_{n-j}(i) = -\frac{\binom{n}{i}}{2^n}\sum_{j=0}^n b_j K_j(i) = - A_i,\]
so $A_i = 0$ for odd $i$.
And if $A_i = 0$ for all odd $i$, then
\[ b_j = \frac{1}{\binom{n}{j}}\sum_{i=0}^n A_i K_j(i) = \frac{1}{\binom{n}{n-j}}\sum_{\substack{0 \leq i \leq n\\i \text{ even}}} A_i K_j(i) = \frac{1}{\binom{n}{n-j}}\sum_{\substack{0 \leq i \leq n\\i \text{ even}}} A_i K_{n-j}(i) = b_{n-j}. \]

\section{Extending and puncturing quasicodes}\label{section: extending and puncturing}

Two common practical operations performed on error-correcting codes are extending and puncturing codes, which respectively increase and decrease the block length by one.
In this section, we generalize both operations to quasicodes, and use these operations to prove \cref{theorem: parity,theorem: symmetry} along with a further refinement of these optima.
\cref{theorem: uniqueness implications} collates all of these results.

For a given code $\mathcal{C}\subseteq\F_2^n$, we may \emph{extend} $\mathcal{C}$ to a new code $\mathcal{C'}\subseteq\F_2^{n+1}$ by adding a parity check bit, i.e., adding a bit that ensures the sum of the bits is always even.
If two codewords in $\mathcal{C}$ differ by an odd number of bits, then in $\mathcal{C'}$ the corresponding codewords differ on their parity check bit as well, increasing their Hamming distance by 1; if two codewords in $\mathcal{C}$ differ by an even number of bits, then their Hamming distance is unchanged by extension.
Hence, if $A$ is the distance distribution of $\mathcal{C}$, then the distance distribution $A'$ of $\mathcal{C}'$ is given by $A_i'=0$ for all odd $i$, and $A_i'=A_i+A_{i-1}$ for all even $i$.
If the minimal distance of $\mathcal{C}$ is $d$, then the minimal distance of $\mathcal{C}'$ is $2\ceil{d/2}$, i.e., rounding $d$ up to the nearest even integer.
We can naturally generalize the definition of extension to quasicodes by applying the same transformation taking $A$ to $A'$.
We will thus call $A'$ the \emph{extension} of the quasicode $A$.
While it is obvious that if $\mathcal{C}$ is a code, then $\mathcal{C'}$ is also a code, it is not immediately clear whether $A$ being a valid quasicode implies $A'$ is a valid quasicode.
However, we will prove that this is indeed true: if $A$ is a feasible solution to the $(n,d)$ Delsarte linear program, then $A'$ is a feasible solution to the $(n+1,2\ceil{d/2})$ Delsarte linear program.
To do so, we first prove the following lemma.

\begin{lemma}\label{lemma: extension lemma}
For $n$ a positive integer, $j \in [n+1]$, and $0 \leq i \leq n$, we have $K_j(2\ceil{i/2};n+1)=K_j(i;n)+K_{n+1-j}(i;n)$.
\end{lemma}
\begin{proof}
If $i$ is even, this immediately follows from \cref{lemma: Krawtchouk recurrence}.
If $i<n$ is odd, then by \cref{lemma: Krawtchouk basic properties} and \cref{lemma: Krawtchouk recurrence},
\begin{align*}
    K_j(i+1;n+1)
    &= (-1)^j K_j(n-i;n+1)
    =(-1)^j (K_j(n-i;n)+K_{j-1}(n-i;n))
    \\ &= K_j(i;n)-K_{j-1}(i;n)
    = K_j(i;n)+K_{n+1-j}(i;n)
\end{align*}
as desired.
If $i=n$ is odd, then
\begin{align*}
    K_j(n+1;n+1)
    &=(-1)^j\binom{n+1}{j}
    =(-1)^j\left(\binom{n}{j}+\binom{n}{j-1}\right)
    =K_j(n;n)-K_{j-1}(n;n)
    \\&=K_j(n;n)+K_{n+1-j}(n;n),
\end{align*}
completing the proof.
\end{proof}

This allows us to prove the extension of a quasicode is a quasicode.

\begin{proposition}\label{proposition: extension quasicode}
If $A$ is a feasible solution to the $(n,d)$ Delsarte linear program, then its extension $A'$ is a feasible solution to the $(n+1,2\ceil{d/2})$ Delsarte linear program.
\end{proposition}
\begin{proof}
The only nontrivial constraints that we need to verify are the Delsarte inequalities.
But \cref{lemma: extension lemma} immediately implies that for all $j \in [n+1]$,
\begin{align*}
    \sum_{i=0}^{n+1} A_i' K_j(i;n+1)
    &= \sum_{i=0}^n A_i K_j(2\ceil{i/2};n+1)
    = \sum_{i=0}^n A_i K_j(i;n) + \sum_{i=0}^n A_i K_{n+1-j}(i;n) \geq 0
\end{align*}
by the Delsarte inequalities for $A$, where if $j=n+1$ the first sum vanishes and the second sum becomes simply $\sum A_i \geq 0$ by the nonnegativity constraints.
\end{proof}

Note that this result implies that the optimal objective value of the $(n-1,2e-1)$ Delsarte linear program is at most the optimal objective value of the $(n,2e)$ Delsarte LP, as any feasible solution of the $(n-1,2e-1)$ LP can be extended to a feasible solution of the $(n,2e)$ LP with the same objective value.

We now prove the reverse direction, that any feasible solution of the $(n,d)$ LP can be mapped to a feasible solution of the $(n-1,d-1)$ LP with the same objective value.
We do so by generalizing the notion of puncturing from codes to quasicodes.
One may \emph{puncture} a code $\mathcal{C}\subseteq\F_2^n$ to yield a new code $\mathcal{C}'\subseteq\F_2^{n-1}$ by simply removing a bit, i.e., an index.
It is easy to see that if $\mathcal{C}$ had minimal distance $d\geq 2$, then $\mathcal{C}'$ is a valid code with minimal distance at least $d-1$.
If $\mathcal{C}$ has support $S$, then $\mathcal{C}'$ has support $S'\subseteq(S\cup(S-1))\setminus\{n\}$, where $S-1:=\{i-1\mid i \in S\}$.
It is important to note that the behavior of puncturing a code depends on which index is removed: especially if $\mathcal{C}$ is not symmetric with respect to its indices, even the distance distribution $A'$ of $\mathcal{C}'$ may depend on the choice of removed index.
Our definition of puncturing a quasicode, which need not be realized by a code, conveniently has no such ambiguity.
If quasicode $A$ has Krawtchouk decomposition $b=(b_0,\dots,b_n)$, then \emph{puncturing} $A$ yields a quasicode $A^P$ given by the truncated Krawtchouk decomposition $b^P=(b_0,\dots,b_{n-1})$.
We now show that if $A$ is a feasible solution to the $(n,d)$ Delsarte linear program and has support $S$, then $A^P$ is a feasible solution to the $(n-1,d-1)$ Delsarte linear program with support $S'\subseteq(S\cup(S-1))\setminus\{n\}$, thus exhibiting the same properties as punctured codes.

\begin{proposition}\label{proposition: puncture quasicode}
If $A$ is a feasible solution to the $(n,d)$ Delsarte linear program for $d \geq 2$ and has support $S$, then its punctured quasicode $A^P$ is a feasible solution to the $(n-1,d-1)$ Delsarte linear program with support $S'\subseteq(S\cup(S-1))\setminus\{n\}$.
\end{proposition}
\begin{proof}
Let $b$ be the Krawtchouk decomposition of $A$.
We will show $b^P$ satisfies the necessary constraints.
Using \cref{lemma: Krawtchouk extension}, for $0 \leq i \leq n-1$,
\begin{align*}
    \sum_{j=0}^{n-1} b_j K_j(i;n-1) 
    &= \frac{1}{2}\sum_{j=0}^{n} b_j (K_j(i;n)+K_j(i+1;n)) - \frac{1}{2}b_n(K_n(i;n)+K_n(i+1;n))
    \\ &= \frac{1}{2}\sum_{j=0}^n b_j(K_j(i;n)+K_j(i+1;n)) - \frac{1}{2}b_n((-1)^i+(-1)^{i+1})
    \\ &= \frac{1}{2}\sum_{j=0}^n b_jK_j(i;n)+\frac{1}{2}\sum_{j=0}^n b_j K_j(i+1;n).
\end{align*}
For $i=0$, as $d \geq 2$ this yields
\[ \sum_{j=0}^{n-1} b_j \binom{n-1}{j} = \frac{1}{2}\sum_{j=0}^n b_j \binom{n}{j} = 2^{n-1}, \]
using the constraints that $b$ must satisfy for the $(n,d)$ linear program.
These constraints for $b$ also imply that for $i \in [d-2]$, this expression equals 0, and more generally for $i \in [n-1]$ it is nonnegative.
In fact, if positive integer $i \not\in (S\cup(S-1))\setminus\{n\}$, then as $A_i=A_{i+1}=0$, this implies $A_i^P = \frac{\binom{n-1}{i}}{2^{n-1}}\sum b_j K_j(i;n-1) = 0$.
Hence, $b^P$ is a feasible solution to the $(n-1,d-1)$ Delsarte linear program with support $S'\subseteq(S\cup(S-1))\setminus\{n\}$.
\end{proof}

\cref{proposition: puncture quasicode} implies that the optimal objective value of the $(n,d)$ Delsarte linear program is at most the optimal objective value of the $(n-1,d-1)$ Delsarte LP, as any feasible solution of the $(n,d)$ LP can be punctured to a feasible solution of the $(n-1,d-1)$ LP with the same objective value, as $b_0$ is unchanged.
Combined with \cref{proposition: extension quasicode}, this implies the $(n-1,2e-1)$ and $(n,2e)$ Delsarte linear programs have the same optimal objective value.
This now enables us to prove \cref{theorem: parity,theorem: symmetry}.

\begin{proof}[Proof of \cref{theorem: symmetry}]
Consider an optimum $A^*$ to the $(n-1,2e-1)$ LP.
By \cref{proposition: extension quasicode}, extending $A^*$ to $A^E$ gives an optimal quasicode to the $(n,2e)$ LP.
In particular, extension creates an even quasicode, and thus the corresponding Krawtchouk decomposition $b^E$ is symmetric.
\end{proof}

To prove \cref{theorem: parity}, we simply puncture $A^E$.

\begin{proof}[Proof of \cref{theorem: parity}]
Puncturing $A^E$ yields an optimum $A^O$ to the $(n-1,2e-1)$ LP, and the corresponding Krawtchouk decompositions $b^E$ and $b^O$ agree on indices $0$ through $n-1$, inclusive.
\end{proof}

In fact, there are further patterns for $A^E$ and $A^O$, as seen in the following result.

\begin{theorem}\label{theorem: uniqueness implications}
For all positive integers $n$ and $e$ such that $2e \leq n$, there exist optima $A^E$ and $A^O$ of the $(n,2e)$ and $(n-1,2e-1)$ Delsarte linear programs, respectively, with corresponding Krawtchouk decompositions $b^E$ and $b^O$, such that the following properties hold:
\begin{enumerate}[label=(\roman*)]
    \item For all $0 \leq j \leq n$, we have $b_j^E=b_{n-j}^E$.
    \item $b^O=(b_0^E,\dots,b_{n-1}^E)$.
    \item For all even $i \in [n]$, we have $A_i^E = A_{i-1}^O + A_i^O$ and $A_{i-1}^O \binom{n-1}{i}=A_i^O\binom{n-1}{i-1}$, where we define $A_i^O=0$ if $i\not\in[0,n-1]$.
    \item If $A^O$ is the unique optimum for $(n-1,2e-1)$, then $A^E$ is the unique optimum for $(n,2e)$.
\end{enumerate}
\end{theorem}
\begin{proof}
We have already shown we may choose optima $A^E$ and $A^O$ so that the first two conditions hold.
We now address the third property.
We trivially have $A_0^E = A_0^O = 1$.
For even $i \in [0,n-1]$, using the Krawtchouk recurrence from \cref{lemma: Krawtchouk recurrence} as well as the symmetry of $b^E$,
\begin{align*}
    \sum_{j=0}^n b_j^E K_j(i;n)
    &= \sum_{j=0}^n b_j^E (K_j(i;n-1)+K_{j-1}(i;n-1))
    \\ &= \sum_{j=0}^{n} \left(b_j^E K_j(i;n-1) + b_{n-j}^E (-1)^i K_{n-j}(i;n-1)\right)
    = 2 \sum_{j=0}^{n-1} b_j^E K_j(i;n-1).
\end{align*}
For odd $i \in [n-1]$, a similar argument yields
\begin{align*}
    2 \sum_{j=0}^{n-1} b_j^E K_j(i;n-1)
    &= \sum_{j=0}^n b_j^E K_j(i;n-1) - \sum_{j=0}^n b_{n-j}^E K_{n-j-1}(i;n-1)
    \\ &= \sum_{j=0}^n b_j^E \left(K_j(i;n-1)-K_{j-1}(i;n-1)\right)
    \\ &= \sum_{j=0}^n (-1)^j b_j^E (K_j(n-i-1;n-1)+K_{j-1}(n-i-1;n-1))
    \\ &= \sum_{j=0}^n (-1)^j b_j^E K_j(n-i-1;n) = \sum_{j=0}^n b_j^E K_j(i+1;n).
\end{align*}
Combining these, we find for even $i \in [n-1]$,
\begin{align*}
    A_{i-1}^O + A_i^O
    &= \frac{1}{2^{n-1}}\left(\binom{n-1}{i-1}\sum_{j=0}^{n-1}b_j^E K_j(i-1;n-1) + \binom{n-1}{i}\sum_{j=0}^{n-1} b_j^E K_j(i;n-1)\right)
    \\ &= \frac{1}{2^n}\left(\binom{n-1}{i-1}\sum_{j=0}^n b_j^E K_j(i;n) + \binom{n-1}{i}\sum_{j=0}^n b_j^E K_j(i;n)\right) = \frac{\binom{n}{i}}{2^n}\sum_{j=0}^n b_j^E K_j(i;n) = A_i^E.
\end{align*}
Additionally, this shows
\[ \frac{A_{i-1}^O}{A_i^O}=\frac{\binom{n-1}{i-1}}{\binom{n-1}{i}},\]
as desired.
If $n$ is odd, the proof is complete; however, if $n$ is even, it remains to show $A_n^E = A_{n-1}^O$ and $A_{n-1}^O \binom{n-1}{n}=A_n^O \binom{n-1}{n-1}$.
The latter can be immediately addressed because both sides equal zero, as $\binom{n-1}{n}=A_n^O=0$.
For the former, as $n-1$ is odd, using the same argument as before gives
\begin{align*}
    A_{n-1}^O = \frac{1}{2^{n-1}}\sum_{j=0}^{n-1} b_j^E K_j(n-1;n-1) = \frac{1}{2^n}\sum_{j=0}^n b_j^E K_j(n;n) = A_n^E,
\end{align*}
completing the proof.

Finally, for the fourth property, if $A^O$ is the unique optimum for $(n-1,2e-1)$, then consider any optimum $A^*$ of the $(n,2e)$ linear program.
Puncturing $A^*$ must yield an optimum for $(n-1,2e-1)$, which can only be $A^O$.
Thus the Krawtchouk decomposition $b^*$ of $A^*$ must satisfy $b^*_i = b_i^O$ for all $0 \leq i \leq n-1$.
The first constraint of the Krawtchouk decomposition LP then fixes $b^*_n$ as well, implying there is a unique optimum to the $(n,2e)$ linear program, which must be $A^E$.
\end{proof}

Hence, extending $A^O$ yields $A^E$, and puncturing $A^E$ yields $A^O$.

The fourth condition implies that if $(n-1,2e-1)$ has a unique optimum, then so does $(n,2e)$; however, the converse does not hold in general.
As shown in \cref{subsection: non-unique primal}, the $(18,6)$ Delsarte linear program has a unique optimum, but $(17,5)$ does not.

The third property of \cref{theorem: uniqueness implications} follows from the symmetry of $b^E$, which requires $b^O$ to satisfy a truncated symmetry.
Constraining the $(n,2e)$ and $(n-1,2e-1)$ Delsarte linear programs to have symmetry and truncated symmetry, respectively, i.e., $b_j = b_{n-j}$, results in a stronger parity phenomenon: these two symmetry-constrained linear programs are equivalent, i.e., have the same feasible region and objective.

\begin{proposition}\label{proposition: symmetry implies parity}
For all positive integers $n$ and $e$ such that $2e \leq n$, the symmetry-constrained $(n,2e)$ and $(n-1,2e-1)$ Delsarte linear programs, where we add the constraint that $b_j=b_{n-j}$ for all $0 \leq j \leq n$, are equivalent.
\end{proposition}
\begin{proof}
We may write the symmetry-constrained $(n,2e)$ Delsarte LP as
\begin{align*}
    \max \quad & b_0 \\
    \text{such that} \quad & \sum_{j=0}^n \binom{n}{j} b_j = 2^n \tag{$\star$}\\
    & \sum_{j=0}^n b_j K_j(i;n) \geq 0 \quad \text{for all } i \in [n] \tag{$\star\star$}\\
    & \sum_{j=0}^n b_j K_j(i;n) = 0 \quad \text{for all } i \in [2e-1] \tag{$\star\star\star$}\\
    & b_j = b_{n-j} \quad \text{for all } j \in \{0,1,\dots,n\} \\
    & b_j \geq 0 \quad \text{for all } j \in [n].
\end{align*}
Similarly, we impose the symmetry condition on the $(n-1,2e-1)$ Delsarte LP, yielding
\begin{align*}
    \max \quad & b_0 \\
    \text{such that} \quad & \sum_{j=0}^{n-1} \binom{n-1}{j} b_j = 2^{n-1} \tag{$\ast$}\\
    & \sum_{j=0}^{n-1} b_j K_j(i;n-1) \geq 0 \quad \text{for all } i \in [n-1] \tag{$\ast\ast$}\\
    & \sum_{j=0}^{n-1} b_j K_j(i;n-1) = 0 \quad \text{for all } i \in [2e-2] \tag{$\ast\ast\ast$}\\
    & b_j = b_{n-j} \quad \text{for all } j \in [n-1] \\
    & b_j \geq 0 \quad \text{for all } j \in [n-1].
\end{align*}
In the $(n-1,2e-1)$ Delsarte LP, we may add the variable $b_n$ with the constraint $b_n = b_0 \geq 0$ without affecting the LP, allowing the symmetry and non-negativity conditions, i.e., the last two constraints, of the two LPs to be identical, using the same decision variables $b_0,\dots,b_n$.

We first show that constraint ($\star$) is equivalent to constraint ($\ast$), and so on.
Using the same argument as in the proof of \cref{theorem: uniqueness implications},
\begin{align*}
    \sum_{j=0}^n \binom{n}{j} b_j
    = \sum_{j=0}^n b_j K_j(0;n)
    = 2 \sum_{j=0}^{n-1} b_j K_j(0;n-1)
    = 2 \sum_{j=0}^{n-1} \binom{n-1}{j}b_j,
\end{align*}
which implies ($\star$) is logically equivalent to ($\ast$).

The ($\star\star$) and ($\star\star\star$) constraints are structurally identical, so we will address them together.
Firstly, for odd $i \in [n]$, notice that by symmetry,
\begin{align*}
    \sum_{j=0}^n b_j K_j(i;n) = \sum_{j=0}^n b_{n-j} (-1)^i K_{n-j}(i;n) = - \sum_{j=0}^n b_j K_j(i;n),
\end{align*}
which implies this sum must equal 0, and thus ($\star\star$) and ($\star\star\star$) trivially hold for odd $i$, and these constraints can be removed from the $(n,2e)$ LP.
For even $i\in[n-1]$, by the proof of \cref{theorem: uniqueness implications},
\begin{align*}
    \sum_{j=0}^n b_j K_j(i;n)
    &= 2 \sum_{j=0}^{n-1} b_j K_j(i;n-1),
\end{align*}
so the ($\star\star$) and ($\star\star\star$) constraints for even $i$ are equivalent for $i\in[n-1]$, and for odd $i \in [n-1]$,
\begin{align*}
    2 \sum_{j=0}^{n-1} b_j K_j(i;n-1)
    &= \sum_{j=0}^n b_j K_j(i+1;n).
\end{align*}
Thus, the constraint for odd $i$ in the $(n-1,2e-1)$ LP is equivalent to the constraint for even $i+1$ in the $(n,2e)$ LP.
Hence, using the symmetry constraints, the ($\star\star\star$) constraints hold for all $i\in[2e-1]$ if and only if they hold for all even $i \in [2e-2]$, which is equivalent to the ($\ast\ast\ast$) constraint holding for all $i \in [2e-2]$.

The ($\star\star$) and ($\ast\ast$) constraints are also equivalent, though we take a bit more care in differentiating the cases depending on the parity of $n$.
If $n$ is even, then the relevant ($\star\star$) constraints are the even $i \in [n]$; the even $i\in[n-2]$ constraints are equivalent to the ($\ast\ast$) constraints for $i \in [n-2]$, and the ($\star\star$) constraint for $i=n$ is equivalent to the ($\ast\ast$) constraint for $i=n-1$, and thus the ($\star\star$) constraints are equivalent to the ($\ast\ast$) constraints.
If $n$ is odd, then we directly see that the relevant ($\star\star$) constraints are the even ones, i.e., even $i \in [n-1]$, which following our previous reasoning are equivalent to the set of ($\ast\ast$) constraints for all $i \in [n-1]$.

This shows that the two symmetry-constrained LPs are equivalent.
\end{proof}

\section*{Data Availability Statement}
All data generated or analyzed during this study are included in this published article.

\section*{Acknowledgements}
We sincerely thank Henry Cohn for his input throughout the research process as well as editing feedback.
We are grateful to the anonymous referees who made helpful suggestions.

\bibliographystyle{plain}
\bibliography{ref}

\end{document}